\documentclass[11pt]{amsart}
\usepackage{amsmath,amssymb,amsthm}
\usepackage{tikz}
\usepackage[letterpaper, margin=1in]{geometry}
\usepackage{todonotes}
\usetikzlibrary{decorations.markings,intersections}
\usepackage{verbatim}
\usepackage{enumerate}
\usepackage{enumitem}
\usepackage{graphicx}

\usepackage[backref]{hyperref}
\hypersetup{
    colorlinks,
    linkcolor={red!60!black},
    citecolor={green!60!black},
    urlcolor={blue!60!black},
    pdftitle={.},
    pdfauthor={.}
}

\newcommand\abs[1]{\ensuremath{\lvert #1\rvert}}

\newcommand{\g}{(G,\gamma)}
\newcommand{\arb}{\operatorname{arb}}
\newcommand{\ord}{\operatorname{ord}}

\usepackage{thmtools, thm-restate}

\newtheorem{theorem}{Theorem}[section]
\newtheorem{lemma}[theorem]{Lemma}
\newtheorem*{THMMAIN}{Theorem~\ref{thm:longcyclepos}}
\newtheorem*{THMMAIN2}{Theorem~\ref{thm:findingkt}}
\newtheorem*{THMMAIN3}{Theorem~\ref{thm:lower}}

\newtheorem{corollary}[theorem]{Corollary}

\newtheorem*{claim*}{Claim}

\usepackage[utf8]{inputenc}

\title{Group vertex-arboricity of group-labelled graphs}
\author{O-joung Kwon}
\author{Xiaopan Lian}

\address[O-joung Kwon]{Department of Mathematics, Hanyang University, Seoul, South Korea and Discrete Mathematics Group, Institute~for~Basic~Science~(IBS), Daejeon,~South~Korea}
\address[Xiaopan Lian]{Center for Combinatorics and LPMC, Nankai University, Tianjin, China}
\email{ojoungkwon@hanyang.ac.kr}
\email{xiaopanlian@mail.nankai.edu.cn}
\thanks{ The first author is supported by the National Research Foundation of Korea (NRF) grant funded by  the Ministry of Science and ICT (No. NRF-2021K2A9A2A11101617 and No. RS-2023-00211670) and by the Institute for Basic Science (IBS-R029-C1). }
\thanks{The second author is supported by National Natural Science Foundation of China (No. 12161141006).}
\date{\today}

\begin{document}
\begin{abstract}
We introduce the vertex-arboricity of group-labelled graphs. For an abelian group $\Gamma$, a $\Gamma$-labelled graph is a graph whose edges are labelled by elements of $\Gamma$. For an abelian group $\Gamma$ and $A\subseteq \Gamma$, the $(\Gamma, A)$-vertex-arboricity of a $\Gamma$-labelled graph is the minimum integer $k$ such that its vertex set can be partitioned into $k$ parts where each part induces a subgraph having no cycle of value in $A$. 

We prove that for every positive integer $\omega$, there is a function $
f_{\omega}:\mathbb{N}\times\mathbb{N}\to \mathbb{R}$ such that if $|\Gamma\setminus A|\le \omega$, then every $\Gamma$-labelled graph with $(\Gamma, A)$-vertex-arboricity at least $f_{\omega}(t,d)$ contains a subdivision of $K_t$ where all branching paths are of value in $A$ and of length at least $d$. 
This extends a well-known result that every graph of sufficiently large chromatic number contains a subdivision of $K_t$, in various directions.

\end{abstract}
\maketitle

\section{Introduction}

The \emph{vertex-arboricity} of a graph $G$ is the minimum integer $k$ such that the vertex set of $G$ is partitioned into $k$ parts so that every part induces a forest. Chartrand, Kronk, and Wall~\cite{Chartrand1968} first introduced this concept, and showed that planar graphs have vertex-arboricity at most $3$. Clearly, a class of graphs has bounded chromatic number if and only if it has bounded vertex-arboricity, as every forest can be partitioned into two independent sets. One may consider a variation of vertex-arboricity by considering cycles with certain properties. For example, the \emph{dichromatic number} of a digraph, which considers partitions into subdigraphs without directed cycles, is a directed analogue of vertex-arboricity. It was introduced by Neumann-Lala~\cite{Neumann1982} in 1982. 

Mader~\cite{Mader} proved that for every positive integer $t$, there is an integer $d(t)$ such that every graph of average degree at least $d(t)$ contains a subdivision of $K_t$. This implies that $K_t$-subdivision free graphs have bounded chromatic number and bounded vertex-arboricity. 
Aboulker et al.~\cite{Aboulker2019} proved a directed analogue that for every positive integer $t$, there is an integer $D(t)$ such that every digraph of dichromatic number at least $D(t)$ contains a subdivision of bidirected complete digraph on $t$ vertices. Thomassen~\cite{Thomassen1983} and Steiner~\cite{Steiner2022} generalized these results by showing that one can further obtain a subdivision of a complete graph/digraph with some modularity constraints.

In this paper, we generalize the vertex-arboricity of graphs to group-labelled graphs. Let~$\Gamma$ be an abelian group. 
A \emph{$\Gamma$-labelled graph} $(G, \gamma)$ is a pair of a graph~$G$ and a function~${\gamma \colon E(G) \to \Gamma}$.
For a $\Gamma$-labelled graph $(G,\gamma)$ and a subset $A$ of $\Gamma$,  the \emph{$(\Gamma, A)$-vertex-arboricity} of $(G, \gamma)$, denoted by $\arb_{(\Gamma, A)}(G,\gamma)$, is the minimum integer $k$ such that the vertex set of $G$ can be partitioned into $k$ parts where each part induces a subgraph having no cycle of $\gamma$-value in $A$, where the $\gamma$-value of a subgraph is the sum of $\gamma(e)$ over all edges $e$ of the subgraph.
If $A=\emptyset$, then this parameter is always at most $1$.  Thus, we will always assume that $A\neq \emptyset$.
In Section~\ref{sec:unbounded}, we observe that the class of all $\Gamma$-labelled graphs has unbounded $(\Gamma, A)$-vertex-arboricity whenever $A\neq \emptyset$.

Since the $(\Gamma, A)$-vertex-arboricity of a $\Gamma$-labelled graph $(G, \gamma)$ is at most the vertex-arboricity of its underlying graph $G$, one can find a subdivision of $K_t$ in every $\Gamma$-labelled graph of sufficiently large $(\Gamma, A)$-vertex-arboricity.
Can we find structures related to the given group-labelling?
In this direction, we raise the following specific questions: 
\begin{itemize}
    \item Can we always find a subdivision of $K_t$ where each branching path has $\gamma$-value in $A$ in every $\Gamma$-labelled graph with sufficiently large $(\Gamma, A)$-vertex-arboricity?
    \item Can we always find a long cycle of $\gamma$-value in $A$ in every $\Gamma$-labelled graph with sufficiently large $(\Gamma, A)$-vertex-arboricity?
\end{itemize}  

Our main results are that both questions are yes when $\Gamma\setminus A$ is finite. Let $(G, \gamma)$ be a $\Gamma$-labelled graph and $A\subseteq \Gamma$. We say that a subdivision of $K_t$ in $G$ is an $(A,d)$-subdivision of $K_t$ if every branching path is of $\gamma$-value in $A$ and of length at least $d$. We say that a cycle in $G$ is an $(A,d)$-cycle if its $\gamma$-value is in $A$ and it has length at least $d$.

\begin{theorem}\label{thm:findingkt}
For every positive integer $\omega$, there is a function $f_{\omega}:\mathbb{N}\times\mathbb{N}\to \mathbb{R}$ satisfying the following. 
Let $\Gamma$ be an abelian group, let $A\subseteq \Gamma$ be a non-empty set with $\abs{\Gamma\setminus A}\le \omega$, and let $(G, \gamma)$ be a $\Gamma$-labelled graph. If $\arb_{(\Gamma, A)}(G, \gamma)\ge f_\omega(t,d)$, then $G$ contains an $(A,d)$-subdivision of $K_t$.
\end{theorem}

\begin{theorem}\label{thm:longcyclepos}
For every positive integer $\omega$, there is a function $g_{\omega}:\mathbb{N}\to \mathbb{R}$ satisfying the following. 
Let $\Gamma$ be an abelian group, let $A\subseteq \Gamma$ be a non-empty set with $\abs{\Gamma\setminus A}\le \omega$, and let $(G, \gamma)$ be a $\Gamma$-labelled graph. If $\arb_{(\Gamma, A)}(G, \gamma)\ge g_\omega(d)$, then $G$ contains an $(A,d)$-cycle.
\end{theorem}

Note that Theorem~\ref{thm:findingkt} does not directly imply Theorem~\ref{thm:longcyclepos}, as we cannot guarantee that the sum of certain values in $A$ is again in $A$. 
We prove in Section~\ref{sec:adsubdivision} that if $\Gamma$ has bounded number of elements of order $2$, then every $(A, 1)$-subdivision of $K_t$ with large $t$ contains an $(A,d)$-cycle. 
Also, Theorem~\ref{thm:findingkt} extends the known fact that every graph of sufficiently large chromatic number contains a subdivision of $K_t$. To see this, we can take $(\Gamma, A)=(\mathbb{Z}, \mathbb{Z}\setminus \{0\})$ and give labels $1$ to all edges in a graph. We point out that the functions in Theorems~\ref{thm:findingkt} and~\ref{thm:longcyclepos} do not depend on the choice of~${\Gamma}$ and~$A$.

We are far from characterizing all pairs $(\Gamma, A)$ that admit such functions, when $\Gamma\setminus A$ is infinite. Sometimes, we can reduce to Theorem~\ref{thm:findingkt} or~\ref{thm:longcyclepos}; for example, if $\Gamma=\mathbb{Z}$ and $A$ is the set of odd integers, then we can equivalently take $(\Gamma, A)=(\mathbb{Z}_2, \{1\})$, which can be seen as taking a quotient group by even integers.
The next theorem tells that for some pair of $\Gamma$ and $A$, the function in Theorem~\ref{thm:longcyclepos} does not exist.

\begin{theorem}\label{thm:lower}
Let $d$ and $t$ be positive integers. Let $\Gamma$ be an abelian group and let $A\subseteq \Gamma$ be a non-empty set such that there is an element $x\in \Gamma$ where $d$ is the unique integer greater than $2$ for which $dx\in A$. Then there is a $\Gamma$-labelled graph $(G, \gamma)$ such that 
\begin{itemize}
    \item $\arb_{(\Gamma, A)}(G, \gamma)\ge t$ and, 
    \item $G$ has no cycle of length at least $d+1$ and $\gamma$-value in $A$.
\end{itemize}
\end{theorem}
For instance, if $\Gamma=\mathbb{Z}$ and $A=\{100\}$, then we can take $x=1$ and $d=100$. By Theorem~\ref{thm:lower}, there is a $\Gamma$-labelled graph of arbitrary large $(\Gamma, A)$-vertex-arboricity which has no cycle with $\gamma$-value in $A$ and of length at least $d+1$.
This means that 
there is no function $g$ satisfying that every $\Gamma$-labelled graph of $(\Gamma, A)$-vertex-arboricity at least $g(t)$ contains a cycle of $\gamma$-value in $A$ and length at least $t$ for all $t\ge 1$.

We present a specific application of Theorem~\ref{thm:findingkt}. Let $G$ be a graph and $F$ be a subset of the edge set of $G$. Let $\eta(G, F)$ be the minimum integer $k$ such that the vertex set of $G$ can be partitioned into $k$ parts where each part induces a subgraph having no cycle $C$ where $|E(C)\cap F|\equiv 1\pmod 2$.
The property of the cycle $C$ can be formulated as a $\mathbb{Z}_2$-labelling with $A=\{1\}$, where edges of $F$ are labelled by $1$ and other edges are labelled by $0$. 
Since $|\Gamma\setminus A|=1$, Theorem~\ref{thm:findingkt} implies that if $\eta(G,F)\ge f_1(t,1)$, then $G$ contains a subdivision of $K_t$, where each branching path contains odd number of edges of $F$.

This paper is organized as follows. 
In Section~\ref{sec:prelim}, we introduce some preliminary concepts including group-labelled graphs. 
In Section~\ref{sec:unbounded}, we show that if $A$ is a non-empty set, then the set of all $\Gamma$-labelled graphs has unbounded $(\Gamma, A)$-vertex-arboricity. 
In Section~\ref{sec:leveling}, we discuss how to find a certain leveling of a graph that is useful to create a subdivision of $K_t$. 
In Section~\ref{sec:longcycle}, we prove Theorems~\ref{thm:longcyclepos} and~\ref{thm:lower}, and in Section~\ref{sec:adsubdivision}, we prove Theorem~\ref{thm:findingkt}.

\vskip 1cm

\section{Preliminaries}\label{sec:prelim}

For an integer~$m$, we write~${[m]}$ for the set of positive integers at most~$m$. For a set $A$ and an integer $m$, we denote by ${A\choose m}$ the set of all subsets of $A$ of size $m$.

In this paper, all graphs are undirected graphs that have no loops and multiple edges, unless stated otherwise.  
Let~$G$ be a graph. 
We denote by~${V(G)}$ and~${E(G)}$ the vertex set and the edge set of~$G$, respectively. 
For a set~$A$ of vertices in~$G$, let~${G[A]}$ be the subgraph of~$G$ induced by~$A$, and let~${G - A}$ be the graph obtained from~$G$ by deleting all the vertices in~$A$ and all edges incident with vertices in~$A$.
If~${A = \{v\}}$, then we write~${G - v}$ for~${G - A}$. 
For two graphs~$G$ and~$H$, let 
$G \cup H := (V(G) \cup V(H), E(G) \cup E(H))$ and $G \cap H := (V(G) \cap V(H), E(G) \cap E(H))$.

A graph~$G$ is a \emph{subdivision} of a graph~$H$ if~$G$ can be obtained from~$H$ by subdividing edges repeatedly. The vertices of $H$ in $G$ are called \emph{branching vertices}, and a path of $G$ whose endpoints are branching vertices and all other vertices are not branching vertices is called a \emph{branching path}.

Let~$A$ and~$B$ be sets of vertices in~$G$. 
An \emph{${(A, B)}$-path} in a graph $G$ is a path from a vertex in~$A$ to a vertex in~$B$ such that all internal vertices are not contained in~${A \cup B}$. 
If $A$ (or $B$) is a singleton set $\{v\}$, then we may replace $A$ (or $B$) with $v$.
An~${(A,A)}$-path of length at least $2$ is shortly denoted as an \emph{$A$-path}. 

A sequence $(L_0, L_1, \ldots, L_p)$ of disjoint sets of
vertices in a connected graph $G$ is called a \emph{leveling} in $G$ if
\begin{itemize}
    \item $\abs{L_0}=1$ and $L_0\cup L_1\cup \cdots \cup L_p=V(G)$, and 
    \item for every $i\in [p]$, every vertex in $L_i$ has a neighbor in $L_{i-1}$ and has no neighbor in $L_0\cup \cdots \cup L_{i-2}$.
\end{itemize}
We can obtain a leveling
 by choosing a vertex $v$, and taking $L_i$ as the set of all vertices at distance $i$ from $v$. The vertex in $L_0$ is called the \emph{starting vertex} of $(L_0, L_1, \ldots, L_p)$.

 Let $K_n$ be the complete graph on $n$ vertices. For positive integers $n_1, \ldots, n_r$, let $R({n_1},\ldots,{n_r})$ be the smallest integer $n$ such that if the edges of $K_n$ is colored by elements in $[r]$, then $K_n$ contains a complete subgraph on $n_i$ vertices whose all edges are colored by $i$ for some $i\in [r]$. The existence of $R({n_1},\ldots,{n_r})$ is guaranteed by Ramsey’s Theorem~\cite{Ramsey1930}. We write $R(n;q)=R(n,n,\ldots, n)$ where $n$ appears $q$ times.

\subsection{Group-labelled graphs}

For every abelian group, we regard its operation as an additive operation and denote its zero by~$0$.

Let $\Gamma$ be an abelian group. Let $\Omega\subseteq\Gamma$. The subgroup of $\Gamma$ generated by $\Omega$ is denoted by  $\langle\Omega\rangle$. If $\Omega=\{g\}$, then we write $\langle g\rangle$ for $\langle\Omega\rangle$. 
The \emph{order} of an element $g\in \Gamma$ is denoted by $\ord_{\Gamma}(g)$. A subgroup $\Lambda$ of $\Gamma$ is \emph{trivial} if $\Lambda=\Gamma$ or $\Lambda=\emptyset$. A \emph{subgraph} of a $\Gamma$-labelled graph~${(G,\gamma)}$ is a $\Gamma$-labelled graph~$(H,\gamma')$ such that~$H$ is a subgraph of~$G$ and~$\gamma'$ is the restriction of~$\gamma$ to~$E(H)$. 
By a slight abuse of notation, we may refer to this $\Gamma$-labelled graph by~$(H,\gamma)$. 

For a $\Gamma$-labelled graph~${(G,\gamma)}$ and a subgraph~${H}$ of ${G}$, let~$\gamma(H)=\sum_{e \in E(H)} \gamma(e)$. We say that the value is the \emph{$\gamma$-value of~$H$}.

For a $\Gamma$-labelled graph~${(G, \gamma)}$ and a subgroup~$\Lambda$ of~$\Gamma$, 
the \emph{induced $(\Gamma/\Lambda)$-labelling} of~$(G,\gamma)$ is 
the~${\Gamma/\Lambda}$\nobreakdash-labelling~$\lambda$ defined by~${\lambda(e) := \gamma(e) + \Lambda}$ for all edges~${e}$ of $G$.

We say that a subdivision of $K_t$ in $G$ is an $(A,d)$-subdivision of $K_t$ if every branching path is of $\gamma$-value in $A$ and of length at least $d$. We say that a cycle in $G$ is an $(A,d)$-cycle if its $\gamma$-value is in $A$ and it has length at least $d$.

We recall that for a $\Gamma$-labelled graph $(G,\gamma)$ and a subset $A$ of $\Gamma$,  the $(\Gamma, A)$-vertex-arboricity of $(G, \gamma)$, denoted by $\arb_{(\Gamma, A)}(G,\gamma)$, is the minimum integer $k$ such that $V(G)$ can be partitioned into $k$ parts where each part induces a subgraph having no cycle of $\gamma$-value in $A$. 

The following directly follows from the definition.

\begin{lemma}\label{lem:comp}
Let $\Gamma$ be an abelian group and let $A\subseteq \Gamma$ be a non-empty set. Let $(G,\gamma)$ be a $\Gamma$-labelled graph. Then $\arb_{(\Gamma, A)}(G, \gamma)=\max\{\arb_{(\Gamma, A)}(G_i, \gamma): G_i \text{ is a connected component of }G\}$.
\end{lemma}

The following lemma tells that whenever we have a leveling $(L_0, L_1, \ldots, L_p)$ of a connected $\Gamma$-labelled graph $(G, \gamma)$, then there is a part $L_i$ having $(\Gamma, A)$-vertex-arboricity at least half of the one of $G$. Together with Lemma~\ref{lem:comp}, we can find a connected component in $G[L_i]$ which has large $(\Gamma, A)$-vertex-arboricity.
This will be used in several places.

\begin{lemma}\label{lem:lel}
Let $\Gamma$ be an abelian group and let $A\subseteq \Gamma$ be a non-empty set. Let $\g$ be a $\Gamma$-labelled graph. If $G$ is a connected graph with a leveling $(L_0, L_1, \ldots,L_p)$, then there are $i\in [p]$ and $Y\subseteq L_i$ such that
\begin{itemize}
\item $G[Y]$ is a connected component of $G[L_i]$, and 
\item $\arb_{(\Gamma,A)}(G[Y],\gamma)\ge \left\lceil\frac{ \arb_{(\Gamma,A)}(G,\gamma)}{2}\right\rceil$. 
\end{itemize}
\end{lemma}
\begin{proof}
Let \[k=\left\lceil\frac{ \arb_{(\Gamma,A)}(G,\gamma)}{2}\right\rceil.\] This implies $\arb_{(\Gamma,A)}(G,\gamma)\in \{2k-1, 2k\}$. Let $(G_1,\gamma)$ and $(G_2,\gamma)$ be the subgraphs of $(G,\gamma)$ induced by the odd levels and even levels, respectively.

We claim that there exists $i\in [p]$ such that $\arb_{(\Gamma,A)}(G[L_i],\gamma)\ge k$.
Suppose for contradiction that for every $i\in [p]$, $\arb_{(\Gamma, A)}(G[L_i],\gamma)\le k-1$. 
Since there is no edge between $L_{i_1}$ and $L_{i_2}$ when $|i_1-i_2|\ge 2$, it implies that $\arb_{(\Gamma, A)}(G_1,\gamma)\le k-1$ and $\arb_{(\Gamma, A)}(G_2,\gamma)\le k-1$. 
Since we can take different sets of colors for $G_1$ and $G_2$, we have $\arb_{(\Gamma, A)}(G,\gamma)\le \arb_{(\Gamma, A)}(G_1,\gamma)+\arb_{(\Gamma, A)}(G_2,\gamma)\le 2k-2$, contradicting the fact that $\arb_{(\Gamma,A)}(G,\gamma)\in \{2k-1, 2k\}$. 

Therefore, there is some $i\in [p]$ such that $\arb_{(\Gamma, A)}(G[L_i],\gamma)\ge k$. 
By Lemma~\ref{lem:comp}, there is a vertex set $Y\subseteq L_i$ such that $G[Y]$ is a connected component of $G[L_i]$ and $\arb_{(\Gamma, A)}(G[Y],\gamma)=\arb_{(\Gamma, A)}(G[L_i],\gamma)\ge k$. This proves the lemma.
\end{proof}
\begin{lemma}\label{lem:deletingvertex}
Let $\Gamma$ be an abelian group and let $A\subseteq \Gamma$ be a non-empty set. Let $\g$ be a $\Gamma$-labelled graph and $v\in V(G)$. Then $\arb_{(\Gamma,A)}(G-v,\gamma)\ge \arb_{(\Gamma,A)}(G,\gamma)-1$.     
\end{lemma}
\begin{proof}
    Suppose for contradiction that $\arb_{(\Gamma,A)}(G-v,\gamma)\le \arb_{(\Gamma,A)}(G,\gamma)-2$. By adding one more part for $v$, we have that $\arb_{(\Gamma,A)}(G,\gamma)\le \arb_{(\Gamma,A)}(G-v,\gamma)+1\le \arb_{(\Gamma,A)}(G,\gamma)-1$, a contradiction. 
\end{proof}

We prove a lemma saying that if a $\Gamma$-labelled graph $(G, \gamma)$ has large $(\Gamma, A)$-vertex-arboricity, then $G$ contains many vertex-disjoint cycles which are of $\gamma$-value in $A$. 

\begin{lemma}\label{lem:disjointcycles1}
Let $t$ be a positive integer.
Let $\Gamma$ be an abelian group and let $A\subseteq \Gamma$ be a non-empty set.
Let $(G, \gamma)$ be a $\Gamma$-labelled graph. 
If $\arb_{(\Gamma, A)}(G, \gamma)\ge 2t$, then $G$ contains $t$ vertex-disjoint cycles whose $\gamma$-values are in $A$.
\end{lemma}
\begin{proof}
We prove by induction on $t$. If $t=1$, then $\arb_{(\Gamma, A)}(G, \gamma)\ge 2$ and $G$ contains a cycle whose $\gamma$-value is in $A$. We assume that $t>1$.

Let $C$ be a shortest cycle of $G$ whose $\gamma$-value is in $A$, and let $v\in V(C)$. 
If $\arb_{(\Gamma, A)}(G-V(C), \gamma)\le 2t-3$, then $\arb_{(\Gamma, A)}(G, \gamma)\le 2t-1$ because both $G[v]$ and $G[V(C)\setminus \{v\}]$ contains no cycles which are of
$\gamma$-value in $A$. This contradicts the assumption that $\arb_{(\Gamma, A)}(G, \gamma)\ge 2t$. Thus, $\arb_{(\Gamma, A)}(G-V(C), \gamma)\ge 2t-2$.

By induction hypothesis, $G-V(C)$ contains $(t-1)$ vertex-disjoint cycles which are of  $\gamma$-value in~$A$. Thus, we obtain $t$ such cycles together with $C$.
\end{proof}

\section{Group-labelled graphs with large $(\Gamma, A)$-chromatic number}\label{sec:unbounded}

In this section, we show that if $A\neq \emptyset$, then  the set of all $\Gamma$-labelled graphs has unbounded $(\Gamma, A)$-vertex-arboricity.

\begin{lemma}
Let $\Gamma$ be an abelian group and let $A\subseteq \Gamma$ be a non-empty set. For every positive integer $t$, there is a $\Gamma$-labelled graph $(G,\gamma)$ with $\arb_{(\Gamma,A)}(G,\gamma)\ge t$. 
\end{lemma}
\begin{proof}
Let $t$ be a positive integer. 

First, we consider the case when there are an element $x\in \Gamma$ and an integer $q\ge 3$ such that $q x\in A$. Let $\ell$ be the smallest integer in $\{q\in \mathbb{N}:q\ge 3,  \,q x\in A\}$. 
Let $G=K_{(\ell-1)({t-1})+1}$ and $\gamma:E(G)\rightarrow \Gamma$ such that $\gamma(e)=x$ for every edge $e\in E(G)$. 

We claim that $\arb_{(\Gamma, A)}(G,\gamma)\ge t$.
On the contrary, suppose that $\arb_{(\Gamma, A)}(G,\gamma)< t$ and so $G$ can be partitioned into at most $t-1$ parts $V_1, \ldots, V_{s}$ where each part induces a subgraph having no cycle of $\gamma$-value in $A$. 
Since $G$ has $(\ell-1)(t-1)+1$ vertices, by the pigeonhole principle, there is a part, say $V_i$, which contains at least $\ell$ vertices. Then $G[V_i]$ has a cycle of length $\ell$, and by the construction, its $\gamma$-value is $\ell x$ which is in $A$. This is a contradiction. Therefore, $\arb_{(\Gamma, A)}(G,\gamma)\ge t$.

Now, we may assume that there is no pair of an element $x\in \Gamma$ and an integer $\ell\ge 3$ for which $\ell x\in A$. 

Since $A\neq\emptyset$, there is an element in $A$. Let $y\in A$. If $\ord_{\Gamma}(y)$ is a positive integer, then there is an integer $\ell\ge 3$ such that $\ell y=y\in A$, contradicting the assumption.
So, $y$ has an infinite order.

Now, we define a $\Gamma$-labelled graph $(G,\gamma)$ with $\arb_{(\Gamma, A)}(G,\gamma)\ge k$ as follows. Let $G=K_{t^2}$ and let $H_1,\ldots,H_t$ be $t$ vertex-disjoint induced subgraphs of $G$, each isomorphic to $K_t$. Let $\gamma:E(G)\rightarrow \Gamma$ such that \begin{itemize}
    \item for every edge $e\in \bigcup_{i\in [t]}E(H_i)$, $\gamma(e)=y$, and 
    \item for every edge $e\in E(K_{t^2})\setminus \bigcup_{i\in [t]}E(H_i)$, $\gamma(e)=0$.
\end{itemize}

Suppose that $\arb_{(\Gamma, A)}(G,\gamma)<t$. 
Consider a partition of $G$ with at most $t-1$ parts, say $V_1, \ldots, V_s$. By the pigeonhole principle, there is a part, say $V_x$, which contains at least $t+1$ vertices. Since each $H_i$ consists of exactly $t$ vertices, $V_x$ intersects at least two subgraphs of $\{H_1, \ldots, H_t\}$. Furthermore, as $V_x$ has at least $t+1$ vertices, there is $j_1\in [s]$ so that $H_{j_1}$ contains two vertices $u$ and $v$ in $V_x$. Let $j_2\in [s]\setminus \{j_1\}$ where $H_{j_2}$ contains an element of $V_x$. Let $w\in V(H_{j_2})$.

Since $\gamma(uvwu)=\gamma(uv)+\gamma(vw)+\gamma(wu)=y\in A$, it means that $G[V_x]$ contains a cycle with $\gamma$-value in $A$, a contradiction. Therefore, $\arb_{(\Gamma, A)}(G,\gamma)\ge t$.
\end{proof}

\section{Lemmas on finding a subset $X$ with special $X$-paths}\label{sec:leveling}

In this section, we prove that if a $\Gamma$-labelled graph $(G, \gamma)$ has sufficiently large $(\Gamma, A)$-arboricity, then one can find a set $S$ of vertices where we can obtain a subdivision of a fixed graph whose branching vertices are in $S$ and branching paths are long.
Furthermore, in Corollay~\ref{cor:longpath3}, we show that if $\Lambda$ is some subgroup of $\Gamma$ where $\Lambda\cap A=\emptyset$, then we can obtain such a subdivision where the $\gamma$-value of each branching path is not in $\Lambda$. This is one of the key ideas of proving Theorem~\ref{thm:findingkt}.

As a starting point, we first argue that we can find a set $X$ of vertices in $G$ for which, $G[X]$ has large $(\Gamma, A)$-arboricity, and for any pair of vertices $x$ and $y$ in $X$, there is a long $X$-path whose endpoints are $x$ and $y$. 

\begin{lemma}\label{lem:longpath}
Let $\ell$ be a positive integer. Let $\Gamma$ be an abelian group and let $A\subseteq \Gamma$ be a non-empty set. Let $\g$ be a connected $\Gamma$-labelled graph. If $\arb_{(\Gamma, A)}(G,\gamma)\ge 2^{\ell}$, then $G$ contains a set $X$ of vertices
such that 
\begin{itemize}
    \item $G[X]$ is connected,
    \item $\arb_{(\Gamma,A)}(G[X],\gamma)\ge \frac{\arb_{(\Gamma,A)}(G,\gamma)}{2^\ell}$, and \item for every pair of distinct vertices $x$ and $y$ in $X$, there is an $X$-path in $G$ of length at least $\ell$ whose endpoints are $x$ and $y$.
\end{itemize}
\end{lemma}
\begin{proof}
We claim that for every integer $0\le j\le \ell$, there are a sequence of vertex subsets $V(G)=X_0\supseteq X_1 \supseteq \cdots \supseteq X_j$ and a sequence of vertices $x_0,x_1,\ldots,x_{j-1}$ such that 
\begin{itemize} 
    \item[(a)] for every $i\in \{0,1,\ldots,j\}$, $G[X_i]$ is connected,
    \item[(b)]  for every $i\in \{0,1,\ldots,j-1\}$, $x_i\in X_i$,
    \item[(c)] for every $i\in [j]$, $\arb_{(\Gamma, A)}(G[X_i],\gamma)\ge \arb_{(\Gamma, A)}(G[X_{i-1}],\gamma)/2$,
    \item[(d)] for every $i\in [j]$ and every vertex $u\in X_i$, there is an $(x_{i-1},u)$-path $P_{x_{i-1},u}$ of length at least 1 in $G$ such that $V(P_{x_{i-1},u})\cap X_i=\{u\}$ and $V(P_{x_{i-1},u})\subseteq X_{i-1}$.
\end{itemize}
 Note that $(X_0=V(G))$ is a sequence for $j=0$. 
 Let $V(G)=X_0\supseteq X_1 \supseteq \cdots \supseteq X_j$ be a sequence of vertex subsets and $x_0,x_1,\ldots,x_{j-1}$ be a sequence of vertices satisfying the conditions (a)--(d) as claimed with $j$ being maximal. 
We show that $j=\ell$. 

 Suppose that $j<\ell$. Note that $G[X_j]$ is a connected graph with \[\arb_{(\Gamma, A)}(G[X_j],\gamma)\ge \frac{\arb_{(\Gamma, A)}(G,\gamma)}{2^j}\ge 2^{\ell-j}.\] 
Let $x_j\in X_j$ be an arbitrary vertex. Let $(L_0, L_1, \ldots, L_p)$ be a leveling of $G[X_j]$ with starting vertex $x_j$. By Lemma~\ref{lem:lel}, there are $q\in [p]$ and $X_{j+1}\subseteq L_q\subseteq X_j$ such that \begin{itemize}
    \item $G[X_{j+1}]$ is a connected component of $L_q$, and 
    \item $\arb_{(\Gamma, A)}(G[X_{j+1}],\gamma)\ge\frac{\arb_{(\Gamma, A)}(G[X_j],\gamma)}{2}$.
\end{itemize}  
Because $(L_0, L_1, \ldots, L_p)$ is a leveling of $G[X_j]$, for every vertex $u\in X_{j+1}\subseteq L_q$, there is a path of length $q$ from $x_j$ to $u$, which contains exactly one vertex from each of $L_0, \ldots, L_q$. So this path is contained in $X_j$.
Thus, we obtain sequences $V(G)=X_0\supseteq X_1\supseteq\cdots\supseteq X_{j+1}$ and $x_0,\ldots,x_j$ satisfying (a)--(d). This contradicts the fact that $j$ is maximal. Therefore, $j=\ell$.

Let $X:=X_\ell$. From the construction of those sequences, we have that $G[X]$ is connected and $\arb_{(\Gamma, A)}(G[X],\gamma)\ge \frac{\arb_{(\Gamma, A)}(G,\gamma)}{2^\ell}$.
Thus, it is sufficient to show that for every pair of distinct vertices $x$ and $y$ in $X$, there is an $X$-path in $G$ of length at least $\ell$ whose endpoints are $x$ and $y$.

Since $x\in X\subseteq X_1$, because of the property $(d)$ with $i=1$, there is an $(x_0,x)$-path $P_{x_0,x}$ of length at least $1$ in $G$ such that $V(P_{x_0,x})\cap X_1=\{x\}$.

Since $y\in X=X_\ell$, because of the property $(d)$ with $i=\ell$,
there is an $(x_{\ell-1},y)$-path $P_{x_{\ell-1},y}$ of length at least $1$ in $G$ such that $V(P_{x_{\ell-1},y})\cap X_\ell=\{y\}$ and $V(P_{x_{\ell-1},y})\subseteq X_{\ell-1}$. 
For every $j\in [\ell-1]$, since $x_j\in X_j$, 
there is an $(x_{j-1},x_j)$-path $P_{x_{j-1},x_j}$ of length at least $1$ where $V(P_{x_{j-1},x_j})\cap X_j=\{x_j\}$ and $V(P_{x_{j-1},x_j})\subseteq X_{j-1}$.

Observe that all paths in $\{P_{x_j,x_{j+1}}:j\in \{0, 1, \ldots, \ell-2\}\}\cup \{P_{x_{\ell-1},y}\}$ are pairwise internally vertex-disjoint. On the other hand, the path $P_{x_0,x}$ is vertex-disjoint from each path in $\{P_{x_j,x_{j+1}}:j\in [\ell-2]\}\cup \{P_{x_{\ell-1},y}\}$.
Let $Q$ be the shortest path from $x_1$ to $x$ in $P_{x_0, x}\cup P_{x_0, x_1}$. Then 
$Q\cup P_{x_1,x_2}\cup\cdots\cup P_{x_{\ell-2},x_{\ell-1}}\cup P_{x_{\ell-1},y}$ is an $X$-path in $G$ whose endpoints are $x$ and $y$. It has length at least $\ell$ since each path in $\{P_{x_j,x_{j+1}}:j\in [\ell-2]\}\cup \{Q, P_{x_{\ell-1},y}\}$ has length at least $1$.
\end{proof}

We apply Lemma~\ref{lem:longpath} recursively to find a subset $S$, where we can obtain a subdivision of a fixed graph such that its branching vertices are on $S$ and its branching paths use different areas. 

\begin{corollary}\label{cor:longpath2}
Let $\ell$ and $m$ be positive integers. Let $\Gamma$ be an abelian group and let $A\subseteq \Gamma$ be a non-empty set. Let $\g$ be a connected $\Gamma$-labelled graph. If $\arb_{(\Gamma, A)}(G, \gamma)\ge 2^{\ell m}$, then there is a sequence $S_0\supseteq S_1\supseteq S_2\supseteq \cdots \supseteq S_m$ of sets of vertices in $G$ such that
\begin{itemize}
    \item for every $i\in [m]$ and every two distinct vertices $x$ and $y$ in $S_m$, there is an $S_m$-path of length at least $\ell$ in $G[S_m\cup (S_{i-1}\setminus S_i)]$ whose endpoints are $x$ and $y$, and 
    \item for every $i\in [m]$, $G[S_i]$ is connected and $\arb_{(\Gamma, A)}(G[S_i],\gamma)\ge \frac{\arb_{(\Gamma, A)}(G,\gamma)}{2^{i\ell }}$.
\end{itemize}
\end{corollary}

\begin{proof}
We claim that for every integer $0\le t\le m$,  there is a sequence of vertex subsets $V(G)=S_0\supseteq S_1\supseteq S_2\supseteq \cdots \supseteq S_t$ such that 
\begin{itemize}
    \item[(i)] for every $i\in [t]$ and every two distinct vertices $x$ and $y$ in $S_t$, there is an $S_t$-path of length at least $\ell$ in $G[S_t\cup (S_{i-1}\setminus S_i)]$ whose endpoints are $x$ and $y$, and 
    \item[(ii)] for every $i\in [t]$, $G[S_i]$ is connected and $\arb_{(\Gamma, A)}(G[S_i],\gamma)\ge \frac{\arb_{(\Gamma, A)}(G,\gamma)}{2^{i\ell }}$.
\end{itemize}
Note that ($S_0=V(G)$) is a sequence for $t=0$. %
Let $V(G)=S_0\supseteq S_1 \supseteq \cdots \supseteq S_t$ be a sequence of vertex subsets satisfying the conditions (i) and (ii) as claimed with $t$ being maximal. 
We show that $t=m$. 
Suppose that $t<m$. Note that $G[S_t]$ is a connected graph with 
\[\arb_{(\Gamma, A)}(G[S_t],\gamma)\ge \frac{\arb_{(\Gamma, A)}(G,\gamma)}{2^{t\ell}}\ge 2^{\ell(m-t)}\ge 2^\ell.\]

We apply Lemma~\ref{lem:longpath} to $G[S_{t}]$ and obtain a set $S_{t+1}\subseteq S_{t}$ such that 
\begin{itemize}
\item $G[S_{t+1}]$ is connected,  
\item $\arb_{(\Gamma, A)}(G[S_{t+1}],\gamma)\ge \frac{\arb_{(\Gamma, A)}(G[S_t],\gamma)}{2^{\ell}}\ge \frac{\arb_{(\Gamma, A)}(G,\gamma)}{2^{(t+1)\ell}}$, and
    \item for every two distinct vertices $x$ and $y$ in $S_{t+1}$, there is an $S_{t+1}$-path $P$ of length at least $\ell$ in $G[S_t]$ whose endpoints are $x$ and $y$.
\end{itemize}

We have that the sequence $V(G)=S_0\supseteq S_1 \supseteq \cdots \supseteq S_t\supseteq S_{t+1}$ satisfies the conditions (i) and (ii) contradicting the fact that $t$ is maximal. Therefore, $t=m$.
\end{proof}

If $\Lambda$ is a subgroup of $\Gamma$ such that $A\cap \Lambda=\emptyset$, then based on Lemma~\ref{lem:longpath} and Corollary~\ref{cor:longpath2}, we can find a vertex set $S$ such that for any two vertex $x,y\in S$ there is a $S$-path with endpoints $x$ and $y$ and are of $
\gamma$-value in $\Gamma\setminus\Lambda$ by using a cycle of $\gamma$-value in $A$.  We formalize this in the following corollary.
\begin{corollary}\label{cor:longpath3}
Let $\ell$ and $m$ be positive integers. Let $\Gamma$ be an abelian group and let $A\subseteq \Gamma$ be a non-empty set. Let $\g$ be a connected $\Gamma$-labelled graph. Assume that there is a sequence $S_0\supseteq S_1\supseteq S_2\supseteq \cdots \supseteq S_m$ of sets of vertices in $G$ such that
\begin{itemize}
    \item[(i)] for every $i\in [m]$ and every two distinct vertices $x$ and $y$ in $S_m$, there is an $S_m$-path of length at least $\ell$ in $G[S_m\cup (S_{i-1}\setminus S_i)]$ whose endpoints are $x$ and $y$, and
    \item[(ii)] $(T_1, T_2)$ is a partition of $S_m$.
\end{itemize}
For each $j\in [m]$, let $L_j:=S_{j-1}\setminus S_j$. Then the following hold.
\begin{enumerate}
    \item For every distinct vertices $x$ and $y$ in $T_1$, every $I\subseteq [m]$ with $|I|=2$,  and every vertex $z$ in $T_2$, there is an $(x,y)$-path $P_{x,y}$ of length at least $2\ell$ such that
$V(P_{x,y})\subseteq \{x,y,z\}\cup  (\bigcup_{i\in I}L_i)$.
    \item Let $\Lambda$ be a non-empty subgroup of $\Gamma$ such that $\Lambda\cap A=\emptyset$. For every distinct vertices $x$ and $y$ in $T_1$, every $I\subseteq [m]$ with $|I|=4$,  and every cycle $C$ contained in $G[T_2]$ of $\gamma$-value in $A$, there is an $(x,y)$-path $P_{x,y}$ of $\gamma$-value  in $\Gamma\setminus \Lambda$ and of length at least $2\ell$ such that
$V(P_{x,y})\subseteq \{x,y\}\cup V(C)\cup  (\bigcup_{i\in I}L_i)$.
\end{enumerate}
\end{corollary}

\begin{proof}
(1)  Let $I=\{i_1, i_2\}$. By Property (i), there are 
\begin{itemize}
    \item an $(x,z)$-path $P_{x,z}$ of length at least $\ell$ with $V(P_{x,z})\subseteq \{x,z\}\cup L_{i_1}$, and
    \item an $(y,z)$-path $P_{y,z}$ of length at least $\ell$ with $V(P_{y,z})\subseteq \{y,z\}\cup L_{i_2}$.
\end{itemize}
Then $P_{x,y}:=P_{x,z}\cup P_{y,z}$ has length at least $2\ell$ and $V(P_{x,y})\subseteq \{x,y,z\}\cup  (\bigcup_{i\in I}L_i)$.
\medskip 

(2) Suppose for contradiction that the claim fails for some pair of vertices of $T_1$, say $x$ and $y$, and some $I\subseteq [m]$ with $|I|=4$ and some cycle $C$ contained in $G[T_2]$ of $\gamma$-value in $A$. 

Let $I=\{i_1,i_2,i_3,i_4\}$. Let $u$ and $v$ be two distinct vertices of $C$. Denote by $Q_1$ and $Q_2$ the two internally vertex-disjoint paths from $u$ to $v$ in $C$.

Since $x,y,u,v\in S_m$, by Property (i), there are 
\begin{itemize}
    \item an $(x,u)$-path $P_{x,u}$ of length at least $\ell$ with $V(P_{x,u})\subseteq \{x,u\}\cup L_{i_1}$, 
    \item an $(x,v)$-path $P_{x,v}$ of length at least $\ell$ with $V(P_{x,v})\subseteq \{x,v\}\cup L_{i_2}$,
    \item a $(y,u)$-path $P_{y,u}$ of length at least $\ell$ with $V(P_{y,u})\subseteq \{y,u\}\cup L_{i_3}$, and
    \item a $(y,v)$-path $P_{y,v}$ of length at least $\ell$ with $V(P_{y,v})\subseteq \{y,v\}\cup L_{i_4}$. 
\end{itemize}  
Let $q_1:=\gamma(Q_1)$, $q_2:=\gamma(Q_2)$, 
$a_1:=\gamma(P_{x,u})$, $a_2:=\gamma(P_{x,v})$, $a_3:=\gamma(P_{y,u})$, and $a_4:=\gamma(P_{y,v})$. 

Since the claim fails at $x$ and $y$,  all the $(x,y)$-paths of length at least $2\ell$ contained in $\{x,y\}\cup V(C)\cup  (\bigcup_{i\in I}L_i)$ should be of $\gamma$-value in $\Lambda$. Thus, we have the following:
$$\begin{aligned}
&a_1+a_3\in \Lambda, \\
&a_2+a_4\in\Lambda, \\
&a_1+q_1+a_4\in\Lambda, \\
&a_2+q_2+a_3\in\Lambda.
\end{aligned}
$$
Hence, $\gamma(C)=q_1+q_2\in \Lambda\subseteq \Gamma\setminus A$, which contradicts the assumption that $\gamma(C)\in A$. Therefore, there is an $(x,y)$-path $P_{x,y}$ of $\gamma$-value in $\Gamma\setminus \Lambda$ and of length at least $2\ell$ such that  $V(P_{x,y})\subseteq \{x,y\}\cup V(C)\cup  (\bigcup_{i\in I}L_i)$.
\end{proof}

\section{Finding a long cycle of $\gamma$-value in $A$}\label{sec:longcycle}

In this section, we prove Theorem~\ref{thm:longcyclepos}.

\begin{THMMAIN}
For every positive integer $\omega$, there is a function $g_{\omega}:\mathbb{N}\to \mathbb{R}$ satisfying the following. 
Let $\Gamma$ be an abelian group, let $A\subseteq \Gamma$ be a non-empty set with $\abs{\Gamma\setminus A}\le \omega$, and let $(G, \gamma)$ be a $\Gamma$-labelled graph. If $\arb_{(\Gamma, A)}(G, \gamma)\ge g_\omega(d)$, then $G$ contains an $(A,d)$-cycle.
\end{THMMAIN}

\begin{proof}
Set $g_\omega(t):=(\omega+2)2^{t(\omega+1)+1}$. By Corollary~\ref{cor:longpath2} with $\ell=\lfloor\frac{t}{2}\rfloor$ and $m=2(\omega+1)$, there is a sequence $S_0\supseteq S_1\supseteq \cdots \supseteq S_m$ of sets of vertices in $G$ such that
\begin{itemize}
    \item[(i)] for every $i\in [m]$ and every two vertices $x$ and $y$ in $S_m$, there is an $S_m$-path of length at least $\ell$ in $G[S_m\cup (S_{i-1}\setminus S_i)]$ whose endpoints are $x$ and $y$, and
    \item[(ii)] for every $i\in [m]$, $G[S_i]$ is connected and $\arb_{(\Gamma, A)}(G[S_i],\gamma])\ge \frac{\arb_{(\Gamma, A)}(G,\gamma)}{2^{i\ell }}$.
\end{itemize}

Observe that
\[\arb_{(\Gamma, A)}(G[S_m],\gamma)\ge \frac{\arb_{(\Gamma, A)}(G,\gamma)}{2^{m\ell }}\ge  \frac{\arb_{(\Gamma, A)}(G,\gamma)}{2^{t(\omega+1) }}\ge2(\omega+2).\] 
By Lemma~\ref{lem:disjointcycles1}, we have that $G[S_m]$ contains $\omega+2$ vertex-disjoint cycles $C_1, C_2, \ldots, C_{\omega+2}$ whose $\gamma$-values are in $A$.
For each $i\in [\omega+2]$, let $a_i$ and $b_i$ be two distinct vertices of $C_i$. 
Denote by $P_i$ and $P'_i$ the two internally vertex-disjoint paths from $a_i$ to $b_i$ contained in $C_i$. For each $j\in [m]$, let $L_j:=S_{j-1}\setminus S_j$. Then $L_i\cap L_j=\emptyset$ for distinct $i,j\in [m]$. 

For every $i\in [\omega+1]$, 
\begin{itemize}
    \item let ${Q}_{i}$ be an $S_m$-path of length at least $\ell$ in $S_m\cup L_i$ whose endpoints are $a_i$ and $a_{i+1}$, and 
    \item let ${R}_{i}$ be an $S_m$-path of length at least $\ell$ in $S_m\cup L_{i+\omega+1}$ whose endpoints are $b_i$ and $b_{i+1}$.
\end{itemize}
Such paths exist by the property (i). Note that the sets in $\{L_1, \ldots, L_m\}$ are pairwise disjoint. Thus, the paths in $\{Q_i:i\in [\omega+1]\}\cup \{R_i:i\in [\omega+1]\}$ are pairwise internally vertex-disjoint.

\begin{figure}[t]
    \centering
    \includegraphics[width=0.5\textwidth]{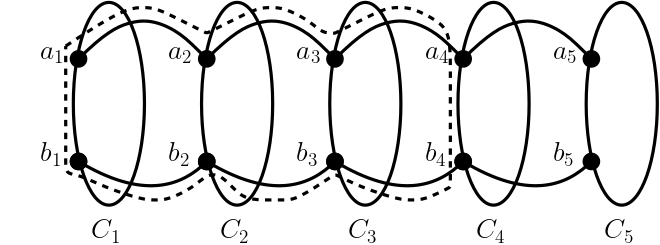}
    \caption{The construction of $C_i, Q_i, R_i, H_i$'s in the proof of Theorem~\ref{thm:longcyclepos}. The dashed cycle represents the cycle $H_3$.}
    \label{fig:cycles}
\end{figure}

For every $i\in [\omega+1]$, let 
\[H_i=P_1\cup P_{i+1}\cup \bigcup_{1\le j\le i}(Q_{j}\cup R_{j}).\]  
See Figure~\ref{fig:cycles} for an illustration.
Note that each $H_i$ has length at least $t$, because it contains at least two paths in $\{Q_i:i\in [\omega+1]\}\cup \{R_i:i\in [\omega+1]\}$. 
If there is some $i\in [\omega+1]$ such that $\gamma(H_i)\in A$, then we are done. Therefore, we may assume that $\gamma(H_i)\in \Gamma\setminus A$ for each $i\in [\omega+1]$.

Since $|\Gamma\setminus A|\le \omega$, by the pigeonhole principle, there are distinct integers $a, b\in [\omega+1]$ such that $\gamma(H_a)=\gamma(H_b)\in \Gamma\setminus A$. Assume $a<b$ without loss of generality. 
Then, we have that 
\[\gamma(P_{a+1})=\gamma\left( P_{b+1}\cup \bigcup_{a+1\le j\le b}(Q_{j}\cup R_{j})\right).\] Let $C=P'_{a+1}\cup  P_{b+1}\cup \bigcup_{a+1\le j\le b}(Q_{j}\cup R_{j})$. Since $P_{a+1}\cup P'_{a+1}=C_{a+1}$, we have that 
\[\gamma(C)=\gamma(P_{a+1}')+\gamma\left( P_{b+1}\cup \bigcup_{a+1\le j\le b}(Q_{j}\cup R_{j})\right)=\gamma(P_{a+1}')+\gamma(P_{a+1})=\gamma (C_{a+1})\in A.\] Furthermore, since $b>a$, the length of $C$ is at least $t$.  Therefore, $C$ is the desired cycle.
\end{proof}

We can apply Theorem~\ref{thm:longcyclepos} when $\Gamma\setminus A=\{0\}$. This implies that a similar function as that in Theorem~\ref{thm:longcyclepos} exists when $\Gamma\setminus A$ is a subgroup of $\Gamma$. We can reduce by taking a quotient group by $\Gamma\setminus A$.

\begin{corollary}\label{cor:cycle1}
There is a function $g:\mathbb{N}\to \mathbb{N}$ satisfying the following.
Let $\Gamma$ be an abelian group and let $A\subseteq \Gamma$ be a non-empty set such that $\Gamma\setminus A$ is a non-trivial subgroup of $\Gamma$. Let $\g$ be a $\Gamma$-labelled graph.  If $\arb_{(\Gamma,A)}(G,\gamma)\ge g(t)$, then $G$ contains a cycle of length at least $t$ and $\gamma$-value in $A$.
\end{corollary}

\begin{proof}
Let $g_w$ be the function defined in the proof of Theorem~\ref{thm:longcyclepos}.  Set $g(t)=g_1(t)$ for all $t\in \mathbb{N}$. 

Let $\Omega=\Gamma\setminus A$, and 
let $\gamma^*$ be the induced $(\Gamma/\Omega)$-labelling of $G$.
Observe that \[\arb_{(\Gamma, A)}(G, \gamma)=\arb_{(\Gamma/\Omega, (\Gamma/\Omega)\setminus \{0+\Omega\})}(G, \gamma^*),\]
because for every induced subgraph $H$ of $G$, $(H, \gamma)$ has a cycle of $\gamma$-value in $A$ if and only if $(H, \gamma^*)$ has a cycle of non-zero $\gamma^*$-value.

Suppose $\arb_{(\Gamma,A)}(G,\gamma)\ge g(t)$. Then $\arb_{(\Gamma/\Omega, (\Gamma/\Omega)\setminus \{0+\Omega\})}(G, \gamma^*)\ge g(t)=g_1(t)$. So, by Theorem~\ref{thm:longcyclepos}, $G$ contains a cycle of length at least $t$ which has non-zero $\gamma^*$-value. Then this cycle has $\gamma$-value in $A$.    
\end{proof}

We prove Theorem~\ref{thm:lower}.

\begin{THMMAIN3}
Let $d$ and $t$ be positive integers. Let $\Gamma$ be an abelian group and let $A\subseteq \Gamma$ such that there is an element $x\in \Gamma$ where $d$ is the unique integer greater than $2$ for which $dx\in A$. Then there is a $\Gamma$-labelled graph $(G, \gamma)$ such that 
\begin{itemize}
    \item $\arb_{(\Gamma, A)}(G, \gamma)\ge t$ and, 
    \item $G$ has no cycle $C$ of length at least $d+1$ with $\gamma(C)\in A$.
\end{itemize}
\end{THMMAIN3}
\begin{proof}
We consider a complete graph $K_n$ where $n=(t-1)(d-1)+1$. Let $\gamma:E(K_n)\to \Gamma$ such that $\gamma(e)=x$ for all $e\in E(K_n)$. Observe that every cycle of $G$ with $\gamma$-value in $A$ has length exactly $d$ by the definition. Thus, $G$ has no cycle with $\gamma$-value in $A$ and of length at least $d+1$. 

Now, we verify that $\arb_{(\Gamma, A)}(G, \gamma)\ge t$.
Suppose that 
$\arb_{(\Gamma, A)}(G, \gamma)< t$. Then for any vertex partition of $G$ into at most $t-1$ parts, there is a part having at least $d$ vertices. By the construction, the part contains a cycle of length $d$, whose $\gamma$-value is in $A$, a contradiction. We conclude that $\arb_{(\Gamma, A)}(G, \gamma)\ge t$.
\end{proof}

For instance, if $\Gamma=\mathbb{Z}$ and $A=\{100\}$, then we can take $x=1$ and $d=100$. By Theorem~\ref{thm:lower}, one can construct a $\Gamma$-labelled graph of arbitrary large $\arb_{(\Gamma, A)}(G, \gamma)$ which has no cycle with $\gamma$-value in $A$ and of length at least $d+1$.
The theorem also says that if $\Gamma=\mathbb{Z}$ and $A=\{x\in \mathbb{Z}:x<0\}\cup \{100\}$, then one can construct a $\Gamma$-labelled graph of of arbitrary large $\arb_{(\Gamma, A)}(G, \gamma)$ which has no cycle with $\gamma$-value in $A$ and of length at least $d+1$, because there are no roles of negative values.

\section{Finding an $(A, d)$-subdivision of $K_t$}\label{sec:adsubdivision}

In this section, we prove Theorem~\ref{thm:findingkt}. 
Let $K_r^t$ be the multigraph obtained from $K_r$ by for each edge $e$ of $K_r$, adding $t-1$ multiple edges of $e$. Note that multigraphs will not appear in the proof, but in the induction step, we will recursively find some construction which is a subdivision of $K_r^t$, where each branching path has length at least $d$.

We first prove the following lemma dealing with a special case.

\begin{lemma}\label{lem:spanA}
    Let $d, t, q, \omega$ be positive integers with $q=t+{t\choose 2}\omega$. Let $\Gamma$ be an abelian group, let $A\subseteq \Gamma$ be a non-empty set with $\abs{\Gamma\setminus A}\le \omega$, and let $g\in \Gamma\setminus A$ where $\langle g \rangle\cap A\neq \emptyset$.
    Let $(H, \gamma)$ be a $\Gamma$-labelled graph that is a subdivision of $K_{q}$ whose branching paths are of $\gamma$-value $g$ and of length at least $d$. Then $H$ contains an $(A,d)$-subdivision of $K_t$.
\end{lemma}
\begin{proof}
    We claim that there is a positive integer $p$ with $p\le \omega+1$ such that $pg\in A$. If $g$ has order at most $\omega$, then $\langle g\rangle=\{0, g, \ldots, (\ord_\Gamma(g)-1)g\}$, and by the assumption that $\langle g\rangle\cap A\neq \emptyset$, such an integer $p$ exists. Otherwise, since $|\Gamma\setminus A|\le \omega$, one of the elements in $\{g, 2g, \ldots, (\omega+1)g\}$ is not contained $\Gamma\setminus A$. So, the claim holds.

Let $w_{1},\ldots,w_{q}$ be the branching vertices of $H$, and for each $\{a,b\}\in {[q]\choose 2}$, let $Q_{\{a,b\}}$ be the branching path of $H$ whose endpoints are $w_a$ and $w_b$.
Let $\phi :{[t]\choose 2}\rightarrow \left[{t\choose 2}\right]$ be a bijection. 
For each $\{a,b\}\in {[t]\choose 2}$ with $a<b$, 
let 
\[
P_{\{a,b\}}:=Q_{\{a,\zeta+1\}}\cup \left(\bigcup_{i\in [p-2]}Q_{\{\zeta+i,~\zeta+i+1\}}\right) \cup Q_{\{\zeta+p-1,~b\}}
, \]
where $\zeta=t+(\phi(\{a,b\})-1)(p-1)$.
Since $q=t+{t\choose 2}\omega$, $P_U$ is well-defined for all $U\in {[t]\choose 2}$.

For $U\in {[t]\choose 2}$,
 $P_U$ is the union of $p$ branching paths of $H$, and therefore, $\gamma(P_U)=pg\in A$ and it has length at least $d$. Moreover, by the construction, for distinct $U_1$ and $U_2$ in ${[t]\choose 2}$, $P_{U_1}$ and $P_{U_2}$ are internally vertex-disjoint.
Hence, $\bigcup_{U\in{[t]\choose 2}}P_U$ is an $(A, d)$-subdivision of  $K_t$-subdivision.
\end{proof}

\begin{THMMAIN2}
For every positive integer $\omega$, there is a function $f_{\omega}:\mathbb{N}\times\mathbb{N}\to \mathbb{R}$ satisfying the following. 
Let $\Gamma$ be an abelian group, let $A\subseteq \Gamma$ be a non-empty set with $\abs{\Gamma\setminus A}\le \omega$, and let $(G, \gamma)$ be a $\Gamma$-labelled graph. If $\arb_{(\Gamma, A)}(G, \gamma)\ge f_\omega(t,d)$, then $G$ contains an $(A,d)$-subdivision of $K_t$.
\end{THMMAIN2}

\begin{proof}
Set 
$$
\begin{aligned}
r_{\omega+1}(t)&:=t+{t\choose 2}(\omega^2+\omega-1),\\
r_{i-1}(t)&:=R(t,R(r_i(t);\omega))\text{ for~each}~i\in [\omega+1] ,\\
c_{i}(t)&:=\sum^{i-1}_{j=0}{r_j(t)\choose 2}\text{ for each }i\in [\omega+1], \text{ and}\\
f_\omega(t,d)&:=(r_0(t)+2c_{\omega+1}(t))2^{2(d+1)c_{\omega+1}(t)}.
\end{aligned}
$$
Clearly, $r_{i-1}(t)\ge r_{i}(t)$ for each $i\in [\omega+1]$.

Let $(G,\gamma)$ be a $\Gamma$-labelled graph with $\arb_{(\Gamma, A)}(G,\gamma)\ge f_\omega(t,d)$. Suppose for contradiction that $G$ does not contain an $(A, d)$-subdivision of $K_t$.

Since $\arb_{(\Gamma, A)}(G,\gamma)\ge f_\omega(t,d)$,  by Corollary~\ref{cor:longpath2} with $\ell:=\lfloor\frac{d+1}{2}\rfloor$ and $m:=4c_{\omega+1}(t)$, there is a sequence $S_0\supseteq S_1\supseteq  \cdots \supseteq S_m$ of sets of vertices in $G$ such that 
\begin{itemize}
    \item[(i)] for every $i\in [m]$ and every two distinct vertices $x$ and $y$ in $S_m$, there is an $S_m$-path of length at least $\ell$ in $G[S_m\cup (S_{i-1}\setminus S_i)]$ whose endpoints are $x$ and $y$, and 
    \item[(ii)] for every $i\in [m]$, $G[S_i]$ is connected and $\arb_{(\Gamma, A)}(G[S_i],\gamma)\ge \frac{\arb_{(\Gamma, A)}(G,\gamma)}{2^{i\ell}}$.
\end{itemize}
For each $j\in [m]$, let $L_j:=S_{j-1}\setminus S_j$. Then $L_i\cap L_j=\emptyset$ for distinct $i,j\in [m]$. 

Observe that
\[\arb_{(\Gamma, A)}(G[S_m],\gamma)\ge \frac{\arb_{(\Gamma, A)}(G,\gamma)}{2^{ m\ell }}\ge r_0(t)+2c_{\omega+1}(t).\] 
Let $T_1=\{v_1,\ldots,v_{r_0(t)}\}$ be a set of $r_0(t)$ vertices in $G[S_m]$ and let $T_2:=S_m\setminus T_1$. Since $\arb_{(\Gamma, A)}(G[S_m], \gamma)\ge r_0(t)+2c_{\omega+1}(t)$, by Lemma~\ref{lem:deletingvertex}, we have that  
    \[\arb_{(\Gamma, A)}(G[T_2],\gamma)\ge \arb_{(\Gamma,A)}(G[S_m],\gamma)-|T_1|\ge 2c_{\omega+1}(t).\] 
    By Lemma~\ref{lem:disjointcycles1}, $G[T_2]$ contains  $c_{\omega+1}(t)$ vertex-disjoint cycles which are of $\gamma$-value in $A$. Let $C_1, \ldots, C_{c_{\omega+1}(t)}$ be such cycles.

Let $\phi :{[r_{0}(t)]\choose 2}\rightarrow \left[{r_{0}(t)\choose 2}\right]$ be a bijection. By Corollary~\ref{cor:longpath3} (1), for each $U=\{a,b\}\in {[r_{0}(t)]\choose 2}$, there is a $(v_a,v_b)$-path $P_U$ of length at least $d$ such that  
\begin{itemize}
\item $V(P_{U})\subseteq \{v_a,v_b\}\cup C_{\phi(U)}\cup (L_{4\phi(U)}\cup L_{4\phi(U)-1}\cup L_{4\phi(U)-2}\cup L_{4\phi(U)-3})$,
\end{itemize}
by picking a vertex of $C_{\phi(U)}$.
Observe that all the paths in $\{P_U:U\in {[r_{0}(t)]\choose 2}\}$ are pairwise internally vertex-disjoint. 
Let $J:=\bigcup_{U\in {[r_{0}(t)]\choose 2}} P_U$.
Then $J$ is a subdivision of $K_{r_0(t)}$ whose branching paths have length at least~$d$.

Since $|\Gamma\setminus A|\le \omega$ and 
$r_{0}(t)=R(t,R(r_1(t);\omega))$, %
by Ramsey's theorem, $J$ contains a subgraph $J_1$ such that
\begin{enumerate}
    \item $J_1$ is either 
\begin{itemize}
    \item a subdivision of $K_{t}$ whose branching paths are of $\gamma$-value in $A$,  or
    \item a subdivision of $K_{r_1(t)}$ whose branching paths are of $\gamma$-value $g$ for some $g\in \Gamma\setminus A$, and
\end{itemize}
\item every branching path of $J_1$ is a branching path of $J$.
\end{enumerate}

As $G$ does not contain an $(A,d)$-subdivision of $K_{t}$,  $J_1$ is a subdivision of $K_{r_1(t)}$ whose branching paths are of $\gamma$-value $g$ for some $g\in \Gamma\setminus A$. If $\langle g\rangle \cap A\neq \emptyset$, then by Lemma~\ref{lem:spanA}, $G$ contains an $(A,d)$-subdivision of $K_t$, a contradiction. Thus, we may assume that 
$\langle g\rangle \cap A= \emptyset$.

Note that $K_{r_1(t)}$ is the same as $K^1_{r_1(t)}$. Thus, $J_1$ is a subdivision of $K^1_{r_1(t)}$ whose branching paths are of $\gamma$-value $g$ for some $g\in \Gamma\setminus A$, and every branching path of $J_1$ is a branching path of $J$. So, every branching path of $J_1$ has length at least $d$.

Let $k$ be the maximum integer such that $G$ contains a subdivision $F$ of $K^k_{r_k(t)}$ with branching vertices $v_1^*, \ldots, v_{r_k(t)}^*$ and there are distinct elements $a_1, \ldots, a_k\in \Gamma\setminus A$ where 
\begin{itemize}
\item [(a)] the branching vertices of $F$ are contained in $T_1$, non-branching vertices are not in $T_1$, and each branching path has length at least $d$,
\item[(b)] $
    V(F) \subseteq  T_1  \cup \left(\bigcup_{i\in [c_{k}(t)]}V(C_i) \right)
      \cup \left(\bigcup_{i\in [4c_{k}(t)]}L_i\right)$,
\item [(c)] for each $U=\{a,b\}\in {[r_k(t)]\choose 2}$ and $j\in [k]$, there is a unique branching path $P^j_U$ of $F$ whose endpoints are $v^*_a$ and $v^*_b$ where $\gamma(P^j_U)=a_j$, and
   
    \item[(d)]  $\langle \{a_1,\ldots,a_{k}\}\rangle\cap A= \emptyset$. 
    \end{itemize}
    Since $\abs{\Gamma\setminus A}\le \omega$, we have $k\le \omega$. 

    Now, we construct one more path between every pair of branching vertices of $F$, which has $\gamma$-value not in  $\langle \{a_1,\ldots,a_{k}\}\rangle$.

Let $\phi^* :{[r_{k}(t)]\choose 2}\rightarrow \left[{r_{k}(t)\choose 2}\right]$ be a bijection. By Corollary~\ref{cor:longpath3} (2) with $\Lambda=\langle \{a_1, \ldots, a_k\}\rangle$, for each $U=\{a,b\}\in {[r_{k}(t)]\choose 2}$, there is a $(v_a^*,v_b^*)$-path $P_U^*$ of $\gamma$-value in $\Gamma\setminus \Lambda$ and of length at least $d$ such that  
\[V(P_{U}^*)\subseteq \{v_a^*,v_b^*\}\cup C_{c_{k}(t)+\phi^*(U)}\cup \left(\bigcup_{i\in [4]}L_{4c_{k}(t)+4\phi^*(U)-4+i}\right).\]
Observe that all the paths in $\{P_U^*:U\in {[r_{k}(t)]\choose 2}\}$ are pairwise internally vertex-disjoint.    
Let $J^*:=\bigcup_{U\in {[r_{k}(t)]\choose 2}} P_U^*$.
Then $J^*$ is a subdivision of $K_{r_k(t)}$. By construction, $V(J^*)\cap V(F)=\{v_1^*, \ldots, v_{r_k(t)}^*\}$.

Since $|\Gamma\setminus A|\le \omega$ and 
$r_{k}(t)=R(t,R(r_{k+1}(t);\omega))$, %
by Ramsey's theorem, $J^*$ contains a subgraph $J_1^*$ such that
\begin{enumerate}
    \item $J_1^*$ is either 
\begin{itemize}
    \item a subdivision of $K_{t}$ whose branching paths are of $\gamma$-value in $A$,  or
    \item a subdivision of $K_{r_{k+1}(t)}$ whose branching paths are of $\gamma$-value $g$ for some $g\in \Gamma\setminus A$, and
\end{itemize}
\item every branching path of $J_1^*$ is a branching path of $J^*$.
\end{enumerate}

As $G$ does not contain an $(A,d)$-subdivision of $K_{t}$,  $J_1^*$ is a subdivision of $K_{r_{k+1}(t)}$ whose branching paths are of $\gamma$-value $g$ for some $g\in \Gamma\setminus A$. Since $g\in \Gamma\setminus \Lambda$, we have $g\notin \langle \{a_1, \ldots, a_k\} \rangle$. 
If $\langle \{a_1, \ldots, a_k, g\} \rangle\cap A= \emptyset$, then by considering $a_{k+1}=g$, $F\cup J_1^*$ is a subdivision of $K^{k+1}_{r_{k+1}(t)}$ satisfying the conditions (a), (b), (c), (d), which contradicts the maximality of $k$. Thus, we may assume that 
$\langle \{a_1, \ldots, a_k, g\} \rangle\cap A\neq \emptyset$. Let $a_{k+1}=g$.
Observe that 
$\ord_\Gamma(a_j)\le \omega$ for each  $j\in [k+1]$; otherwise, $\langle a_j\rangle$ must intersect $A$. 

Let $F^*:=F\cup J_1^*$.
Now, we finish the proof by showing that $F^*$ contains an $(A,d)$-subdivision of~$K_t$. Note that $F^*$ is a subdivision of $K_{r_{k+1}(t)}^{k+1}$. Let $\{z_1, \ldots, z_{r_{k+1}(t)}\}$ be the set of branching vertices of $F^*$.
By construction, 
for each $U=\{a,b\}\in {[r_{k+1}(t)]\choose 2}$ and $j\in [k+1]$, there is a unique branching path $Q^j_U$ of $F^*$ whose endpoints are $z_a$ and $z_b$ where $\gamma(Q^j_U)=a_j$.

Let $s\in \langle \{a_1,\ldots,a_{k+1}\}\rangle\cap A$. Note that $s\neq 0$, as $\langle a_1\rangle \cap A\neq \emptyset$ and $a_1$ has order at most $\omega$.
Thus, there are integers $p_1,\ldots,p_{k+1}$, where $p_j\le \omega$ for each $j\in [k+1]$ and at least one of them is nonzero, such that \[s=\sum^{k+1}_{j=1}p_j a_j.\] 
Let $\lambda:=\sum^{k+1}_{j=1}p_j$. Note that $\lambda\le (\omega+1)\omega$.

Let $\theta :{[t]\choose 2}\rightarrow \left[{t\choose 2}\right]$ be a bijection. For each $U=\{a,b\}\in {[t]\choose 2}$ with $a<b$, let $R_U$ be a path in $F^*$ from $z_a$ to $z_b$ such that 
\begin{itemize}
    \item $V(R_U)\cap \{z_i:i\in [r_{k+1}(t)]\}=\{z_a, z_b\}\cup \{z_i:t+(\theta(U)-1)(\lambda-1)+1\le i\le t+\theta(U)(\lambda-1)\}$, 
    \item for each $j\in [k+1]$, $R_U$ contains exactly $p_j$ many branching paths of $F^*$ of $\gamma$-value $a_j$.
\end{itemize}
Since $r_{k+1}(t)\ge r_{\omega+1}(t)=t+{t\choose 2}(\omega^2+\omega-1)\ge t+{t\choose 2}(\lambda-1)$, 
we can choose such a path $R_U$ for all $U\in {[t]\choose 2}$.
For each $U\in {[t]\choose 2}$, $\gamma(R_U)=\sum^{k+1}_{j=1}p_ja_j=s\in A.$ Moreover, by the construction, for distinct $U_1$ and $U_2$ in ${[t]\choose 2}$, $R_{U_1}$ and $R_{U_2}$ are internally vertex-disjoint.  Therefore, $\bigcup_{U\in{[t]\choose 2}}R_U$ is an $(A,d)$-subdivision of $K_t$, a contradiction.
\end{proof}

\begin{corollary}\label{cor:subdivision1}
Let $f_\omega$ be the function defined in Theorem~\ref{thm:findingkt}.
Let $\Gamma$ be an abelian group and let $A\subseteq \Gamma$ be a non-empty set such that $\Gamma\setminus A$ is a non-trivial subgroup of $\Gamma$. Let $\g$ be a $\Gamma$-labelled graph.  If $\arb_{(\Gamma,A)}(G,\gamma)\ge f_1(t,d)$, then $G$ contains an $(A, d)$-subdivision of $K_t$.
\end{corollary}

\begin{proof}
Let $f_w$ be the function defined in the proof of Theorem~\ref{thm:findingkt}.  Set $f(t,d)=f_1(t,d)$ for all $(t,d)\in \mathbb{N}\times \mathbb{N}$. 
Let $\Omega=\Gamma\setminus A$, and 
let $\gamma^*$ be the induced $(\Gamma/\Omega)$-labelling of $G$. 
For every induced subgraph $H$ of $G$, $(H, \gamma)$ has a cycle of $\gamma$-value in $A$ if and only if $(H, \gamma^*)$ has a cycle of non-zero $\gamma^*$-value. This implies that 
 \[\arb_{(\Gamma, A)}(G, \gamma)=\arb_{(\Gamma/\Omega, (\Gamma/\Omega)\setminus \{0+\Omega\})}(G, \gamma^*).\]

Suppose $\arb_{(\Gamma,A)}(G,\gamma)\ge f_1(t,d)$. Then $\arb_{(\Gamma/\Omega, (\Gamma/\Omega)\setminus \{0+\Omega\})}(G, \gamma^*)\ge f_1(t,d)$. By Theorem~\ref{thm:findingkt}, $G$ contains an $( (\Gamma/\Omega)\setminus \{0+\Omega\}, d)$-subdivision $H$ of $K_t$. Thus $H$ is an $(A, d)$-subdivision of $K_t$. 
\end{proof}

Now, we show that if $\Gamma$ has bounded number of elements of order $2$, then every $(A, 1)$-subdivision of $K_t$ with large $t$ contains an $(A,k)$-cycle. 
\begin{theorem}
For positive integers $p$ and $\omega$, there is a function $f_{(\omega,p)}:\mathbb{N}\to \mathbb{R}$ satisfying the following. 
Let $\Gamma$ be an abelian group with $|\{g:\ord_\Gamma(g)\le 2\}|\le p$. Let $A\subseteq \Gamma$ be a non-empty set with $\abs{\Gamma\setminus A}\le\omega$. Let $(G, \gamma)$ be a $\Gamma$-labelled graph. If $t\ge f_{(\omega,p)}(k)$, then every $(A, 1)$-subdivision of $K_t$ in $G$ contains an $(A, k)$-cycle. 
\end{theorem}
\begin{proof}
Let 
$$ 
\begin{aligned}
\mu:&=\max\{(\omega+2)^2,2\omega+6\},\\
\beta:& =R(\mu;p),\\
r:&=R(\beta;\omega^3),\\
f_{(\omega,p)}(k):&=(r+k)(\omega+2). 
\end{aligned}
$$

Let $F$ be an $(A, 1)$-subdivision of $K_t$ contained in $G$ with $t\ge f_{(\omega,p)}(k)$. Let $\{u_1, \ldots, u_{(r+k)(\omega+2)}\}$ be a set of $(r+k)(\omega+2)$ branching vertices of $F$. Then $F$ contains 
$\omega+2$ vertex-disjoint $(A, 1)$-subdivisions of $K_r$, say $F_1,\ldots,F_{\omega+2}$, where their branching vertices are contained in $\{u_i:i\in [r(\omega+2)]\}$.

We claim that for every $H\in \{F_1,\ldots,F_{\omega+2}\}$, we have  $\arb_{(\Gamma,A)}(H,\gamma)\ge 2$, that is, $H$ contains a cycle of $\gamma$-value in $A$. Suppose for contradiction that  $\arb_{(\Gamma,A)}(H,\gamma)=1$.

Let $\{v_1,\ldots,v_r\}$ be the set of branching vertices of $H$. Let $P_{\{a,b\}}$ be the branching path in $H$ between $v_a$ and $v_b$ for each $\{a,b\}\in {[r]\choose 2}$. Then 
for each $\{x,y\} \in {[r]\choose 2}$ with $\{x,y\}\cap\{a,b\}=\emptyset$, we have that 
\begin{align*}
    2\gamma(P_{\{a,b\}})&=\gamma\left(P_{\{a,b\}}\cup P_{\{b,x\}}\cup P_{\{x,a\}}\right)+\gamma\left(P_{\{a,b\}}\cup P_{\{b,y\}}\cup P_{\{y,a\}}\right)\\
    &-\gamma\left(P_{\{y,a\}}\cup P_{\{a,x\}}\cup P_{\{x,b\}}\cup P_{\{b,y\}}\right).
\end{align*}
Let \[\mathcal{S}:=\left\{s_1+s_2-s_3: s_i\in \Gamma\setminus A \text{~for ~each}~i\in [3]\right\}.\]  Then $2\gamma(P_{U})\in \mathcal{S}$ for each $U\in {[r]\choose 2}$. Since $|\Gamma\setminus A|\le \omega$, we have that $|\mathcal{S}|\le \omega^3$. 

 Since
$r:=R(\beta;\omega^3)$,
by Ramsey's theorem, $H$ contains a subgraph $H_1$ which is a subdivision of $K_{\beta}$ with branching vertices $w_1,\ldots,w_{\beta}$ satisfying that 
\begin{itemize}
\item each branching path of $H_1$ is a branching path of $H$, and
\item $2\gamma(Q_{U_1})=2\gamma(Q_{U_2})$ for distinct $U_1$ and $U_2$ in ${[\beta]\choose 2}$.
\end{itemize}
where for each $\{a,b\}\in {[\beta]\choose 2}$, $Q_{\{a,b\}}$ is the branching path of $H_1$ between $w_a$ and $w_b$.  

Let $B:=\{g:\ord_\Gamma(g)\le 2\}$. Let $s\in \{\gamma(P_U):U\in {[\beta]\choose 2}\}$, and for each $U\in {[\beta]\choose 2}$,
let $g_U=\gamma(Q_U)-s$. Observe that $g_U\in B$, because $2g_U=2\gamma(Q_u)-2s=0$.
Since
$\beta:=R(\mu;p)$,
by Ramsey's theorem, $H_1$ contains a subgraph $H_2$ which is a subdivision of $K_{\mu}$ with branching vertices $z_1,\ldots,z_{\mu}$ satisfying that there is $g\in B$ such that 
\begin{itemize} 
\item each branching path of $H_2$ is a branching path of $H_1$ and has $\gamma$-value $s+g\in A$.
\end{itemize}
For each $\{a,b\}\in {[\mu]\choose 2}$, let $Q^*_{\{a,b\}}$ be the branching path of $H_2$ between $z_a$ and $z_b$.

Let $q\in \{3,4,\ldots, \omega+4\}$. Let $U=\{a,b\}\in {[\mu]\choose 2}$.  Since $\mu\ge(\omega+2)^2$, there are $\omega+1$ pairwise internally vertex-disjoint $(z_a,z_b)$-paths, say $R_1,\ldots,R_{\omega+1}$, where each path contains exactly $q$ branching vertices of $H_2$. For each $i\in [\omega+1]$, let $C_i:=Q^*_U\cup R_i$.
Since $\arb_{(\Gamma,A)}(H,\gamma)=1$, we have that $\gamma(C_i)\in \Gamma\setminus A$ for each $i\in [\omega+1]$.  Since $|\Gamma\setminus A|\le \omega$, there are $C_i$ and $C_j$ such that $\gamma(C_i)=\gamma(C_j)\in \Gamma\setminus A$ for some distinct $i,j\in [\omega+1]$ which implies that
\[\gamma(R_i)=\gamma(R_j).\]
Observe that $R_i\cup R_j$ forms a cycle in $H_2$ which contains exactly $2(q-1)$ branching vertices. Thus \[\gamma(R_i\cup R_j)=\gamma(R_i)+\gamma(R_j)=2\gamma(R_j)=2(q-1)(s+g) \in \Gamma\setminus A.\]
Therefore, we summarize that $\{2(q-1)(s+g):3\le q\le \omega+4\}\subseteq \Gamma\setminus A$. We finish the proof by distinguishing the order of $s+g$ in $
\Gamma$.

Suppose that $\ord_{\Gamma}(s+g)\ge 2(\omega+3)$. Then $|\{2(k-1)(s+g):3\le k\le \omega+4\}|=\omega+1>\omega=|\Gamma\setminus A|$, a contradiction. Hence, assume that $\ord_\Gamma(s+g)\le 2\omega+5$. 

Since $\mu\ge 2\omega+6$, $H_2$ contains a cycle $C$ which contains exactly $(\ord_\Gamma(s+g)+1)$ branching vertices. We have that $\gamma(C)=(\ord_\Gamma(s+g)+1)(s+g)=s+g\in A$ contradicting the fact that $\arb_{(\Gamma,A)}(H,\gamma)=1$.

Let $D_1, \ldots, D_{\omega+2}$ be the cycles of $\gamma$-values in $A$, that are contained in $F_1, , \ldots, F_{\omega+2}$, respectively. 
For each $i\in [\omega+2]$, let $a_i$ and $b_i$ be two distinct vertices of $D_i$. 
Denote by $X_i$ and $X'_i$ the two internally vertex-disjoint paths from $a_i$ to $b_i$ contained in $C_i$.

Let $\ell=\lfloor \frac{k}{2}\rfloor$.
For every $i\in [\omega+1]$, 
\begin{itemize}
    \item let ${Y}_{i}$ be a path of $F$ where its endpoints are $a_i$ and $a_{i+1}$ and it contains exactly $\ell+1$ branching vertices of $F$,
    \item let ${Z}_{i}$ be a path of $F$ where its  endpoints are $b_i$ and $b_{i+1}$ and it contains exactly $\ell+1$ branching vertices of $F$, 
    \item $Y_i$ and $Z_i$ are vertex-disjoint, and the branching vertices of $F$ in $Y_i\cup Z_i-\{a_i, a_{i+1}, b_i, b_{i+1}\}$ are contained in $\{u_j:r(\omega+2)+(i-1)k+1\le j\le r(\omega+2)+ik\}$.
\end{itemize}
Note that the paths in $\{Y_i:i\in [\omega+1]\}\cup \{Z_i:i\in [\omega+1]\}$ are pairwise internally vertex-disjoint.

For every $i\in [\omega+1]$, let 
\[H_i=X_1\cup X_{i+1}\cup \bigcup_{1\le j\le i}(Y_{j}\cup Z_{j}).\]  
Note that each $H_i$ has length at least $k$, because it contains at least two paths in $\{Y_i:i\in [\omega+1]\}\cup \{Z_i:i\in [\omega+1]\}$. 
If there is some $i\in [\omega+1]$ such that $\gamma(H_i)\in A$, then we are done. Therefore, we may assume that $\gamma(H_i)\in \Gamma\setminus A$ for each $i\in [\omega+1]$.

Since $|\Gamma\setminus A|\le \omega$, by the pigeonhole principle, there are distinct integers $a, b\in [\omega+1]$ such that $\gamma(H_a)=\gamma(H_b)\in \Gamma\setminus A$. Assume $a<b$ without loss of generality. 
Then, we have that 
\[\gamma(X_{a+1})=\gamma\left( X_{b+1}\cup \bigcup_{a+1\le j\le b}(Y_{j}\cup Z_{j})\right).\] Let $C=X'_{a+1}\cup  X_{b+1}\cup \bigcup_{a+1\le j\le b}(Y_{j}\cup Z_{j})$. Since $X_{a+1}\cup X'_{a+1}=D_{a+1}$, we have that 
\[\gamma(C)=\gamma(X_{a+1}')+\gamma\left( X_{b+1}\cup \bigcup_{a+1\le j\le b}(Y_{j}\cup Z_{j})\right)=\gamma(X_{a+1}')+\gamma(X_{a+1})=\gamma (D_{a+1})\in A.\] Furthermore, since $b>a$, the length of $C$ is at least $k$.  Therefore, $C$ is the desired cycle.
\end{proof}

\end{document}